\documentclass{aims}
\usepackage{amsmath}
  \usepackage{paralist}
  \usepackage{graphics} 
  \usepackage{epsfig} 
\usepackage{subfigure,graphicx}  \usepackage{epstopdf}
 \usepackage[colorlinks=true]{hyperref}
 \usepackage{cases}
 \usepackage{algorithm}
 \usepackage{algpseudocode}
 \usepackage[pagewise]{lineno}
\hypersetup{urlcolor=blue, citecolor=red}

\def\currentvolume{X}
 \def\currentissue{X}
  \def\currentyear{200X}
   \def\currentmonth{XX}
    \def\ppages{X--XX}

\newtheorem{theorem}{Theorem}[section]

\newtheorem{lemma}[theorem]{Lemma}
\newtheorem{proposition}{Proposition}

\theoremstyle{definition}
\newtheorem{definition}[theorem]{Definition}
\newtheorem{remark}{Remark}

\title[A Bayesian level set method for an IMSP in acoustics] 
      {A Bayesian level set method for an inverse medium scattering problem in acoustics}

\author[J. Huang and Z. Deng and L. Xu]{}


 \email{jfhuang00@163.com}
 \email{dengzhl@uestc.edu.cn}
 \email{xul@uestc.edu.cn}

\thanks{$^*$ Corresponding author: xul@uestc.edu.cn}

\begin{document}
\maketitle

\centerline{\scshape Jiangfeng Huang}
\medskip
{\footnotesize
 \centerline{School of Mathematical Sciences, University of Electronic Science and Technology of China,}
   \centerline{Sichuan 611731, China}
} 

\medskip

\centerline{\scshape Zhiliang Deng}
\medskip
{\footnotesize
 \centerline{School of Mathematical Sciences, University of Electronic Science and Technology of China,}
  \centerline{Sichuan 611731, China}
}
\medskip

\centerline{\scshape Liwei Xu$^*$}
\medskip
{\footnotesize
 \centerline{School of Mathematical Sciences, University of Electronic Science and Technology of China,}
  \centerline{Sichuan 611731, China}
}

\bigskip


\begin{abstract}
In this work, we are interested in the determination of the shape of the scatterer for the two dimensional time harmonic inverse medium scattering problems in acoustics. The scatterer is assumed to be a piecewise constant function with a known value inside inhomogeneities, and its shape is represented by the level set functions for which we investigate the information using the Bayesian method. In the Bayesian framework, the solution of the geometric inverse problem is defined as a posterior probability distribution. The well-posedness of the posterior distribution would be discussed, and the Markov chain Monte Carlo (MCMC) methods will be applied to generate samples from the arising posterior distribution. Numerical experiments will be presented to demonstrate the effectiveness of the proposed method.
\end{abstract}

\section{Introduction}

The inverse scattering problems have been extensively investigated because of their great importance and broad applications in radar and sonar, geophysical exploration, medical imaging, and to name a few \cite{Colton2019,Colton2013}. One of the main goals of the inverse scattering problems is to determine the unknown scatterer, such as location, geometry, or material property etc. \cite{Hohage,Natterer,Vogeler}. This kind of inverse scattering problems can be regarded as inverse medium scattering problems (IMSP). The IMSP are ill-posed and nonlinear admitting great theoretical and computational challenges, which have attracted attention of many researchers in past decades.

 There have been many numerical methods being proposed to solve the IMSP \cite{Potthast,Bao3,Bao2012}. Classical methods for the IMSP can be roughly classified into two categories: direct methods and indirect methods. The direct methods mainly recover the support or the shape of the scatterer, such as linear sampling methods \cite{Cheney, Colton2003, Haddar}, multiple signal classification methods \cite{Hou2006, Gruber}, and factorization methods \cite{Kirsch, Kirsch2008}. The indirect methods attempt to determine the unknown representation of the scatterer by applying regularization techniques, including recursive linearization methods \cite{Bao2015, Bao2005}, level set methods \cite{Dorn,Santosa}, and Gauss-Newton methods \cite{Kress2003,Kress1994}. Among these indirect methods, the level set method is a good methodology for the computation of evolving boundaries and interfaces \cite{Dorn2009}. The level set method was originally designed to track propagating interfaces through topological changes \cite{Osher}, and more recently it has been found applications in inverse problems involving obstacles \cite{Santosa, Tai, Ito}.

In addition to the classical methods, another kind of methods lies in the class of statistical methods, and one of those is known as the Bayesian approach. The Bayesian method has attracted considerable attention for inverse problems due to its ability of uncertainty quantification \cite{Dashti, Stuart, Kaipio}. Recently, it has been widely applied to solve the inverse scattering problems \cite{Wang2015, Jia2018, Huang2020, Li2020, Li2019, Liu2019}. In the Bayesian setting, the Gaussian measures are favorable options of the prior distributions, which play central roles in the theory of the Bayesian approach \cite{Stuart}. Samples from the Gaussian priors are generated by solving a related stochastic differential equation or using the Karhunen-Lo\`{e}ve expansion based on the eigenfunctions and eigenvalues of covariance operators of the prior distributions \cite{Iglesias, Dunlop, Jiang, Li2015}. The solution to the Bayesian inversion, i.e. the original problem, is a posterior distribution. To explore the information of the posterior distribution, sampling methods such as the Markov chain Monte Carlo (MCMC) methods are usually employed \cite{Cotter, Feng2018, Chi2001}.

In this work, we are mainly interested in solving the IMSP by the Bayesian level set method, which is a coupling of the Bayesian method and the level set method. Assuming that the scatterer is a piecewise constant function with known values, we characterize the shape of the scatterer by the level set functions. There are few literatures on the numerical solution of the inverse scattering problems by using the Bayesian level set method. In \cite{Iglesias}, it establishes the mathematical foundations of the Bayesian level set method, and its hierarchical extension has been developed in \cite{Dunlop}. In \cite{Deng2,Deng}, the Bayesian method and the ensemble Kalman filter approach based on level set parameterization are introduced for acoustic source identification using multiple frequency information, respectively.  Actually, when the level set method is coupled with the Bayesian approach, there are several advantages for the shape reconstruction. First of all, there are no needs on the implementation of the Fr\'{e}chet derivative of the forward map as well as the corresponding adjoint operator. Secondly, one needs no considering the evolution of the level set functions governed by a Hamilton-Jacobi type equation. Finally, the Bayesian level set method not only can provide with point estimates of the solution,  such as the maximum a posterior (MAP) estimate and the conditional mean (CM) estimate,  but also can provide with a systematic framework for quantifying the uncertainty.  In this paper, we consider the Whittle-Mat\'{e}rn Gaussian random fields as the prior\cite{Roininen,Lindgren}, with which the level sets of the Gaussian random fields have Lebesgue measure zero \cite{Iglesias}. We will also discuss the well-posedness of the posterior distribution based on Bayes' theorem. Applying the MCMC methods, we will show the numerical results via the CM estimates.

The rest of the paper is organized as follows. In Section 2, we simply descible the forward model, and employ the Dirichlet-to-Neumann finite element method (DtN-FEM) as the forward solver \cite{Hsiao,Geng2017}. We discuss the Bayesian level set approach solving the IMSP with the proposed prior and the well-posedness theory of the posterior distribution in Section 3. In Section 4, the numerical results are presented to illustrate the effectiveness of the proposed method.

\section{Direct Scattering Problem}
\label{sec:2}
In this section, we introduce the propagation of time harmonic acoustic waves in two dimensions. The scatterer formed by an inhomogeneous medium is embedded in an infinite homogeneous background medium.
\subsection{A Model Problem}
The scattering problem under consideration is modeled by
\begin{subequations}
\begin{numcases}{}
  \Delta u+k^2(1+q(x))u =0, \quad \text{in}\quad\mathbb{R}^{2}, \label{EQ_1a} \\
  \lim_{r\rightarrow \infty}r^{\frac{1}{2}}(\frac{\partial u^{\rm s}}{\partial r}-iku^{\rm s})=0,\quad r=|x|, \label{EQ_1c}
\end{numcases}
\end{subequations}
where $k>0$ denotes the wavenumber, $u =u^{\rm s}+u^{\rm i}$ is the total field, $u^{\rm i}$ is the plane incident field, and $u^{\rm s}$ is the scattered field which satisfies the Sommerfeld radiation condition (\ref{EQ_1c}) uniformly in all directions. $q(x)>-1$ is assumed to be a piecewise constant function describing the scatterer. Moreover, we assume that the scatterer has a compact support contained in $B_{R}:=\{x \in\mathbb{R}^{2}:|x|<R\}$, which is bounded by an artificial boundary $\Gamma_{R}:=\{x\in\mathbb{R}^{2}:|x|=R\}$ with $R$ being sufficiently large to enclose the scatterer inside (see Fig. \ref{f1}). In particular, considering the plane incident wave $u^{\rm i}=e^{ikx\cdot \textbf{d}}$ with the incident direction $\textbf{d}\in \{x\in\mathbb{R}^{2}:|x|=1\}$, we can write the equation (\ref{EQ_1a}) into
\begin{equation}\label{eq3}
\Delta u^{\rm s}+k^{2}(1+q(x))u^{\rm s}=-k^{2}q(x)u^{\rm i}\quad \text{in}\quad \mathbb{R}^{2}.
\end{equation}
\begin{figure}
  \centering
  \includegraphics[width=1.6in,height=1.4in]{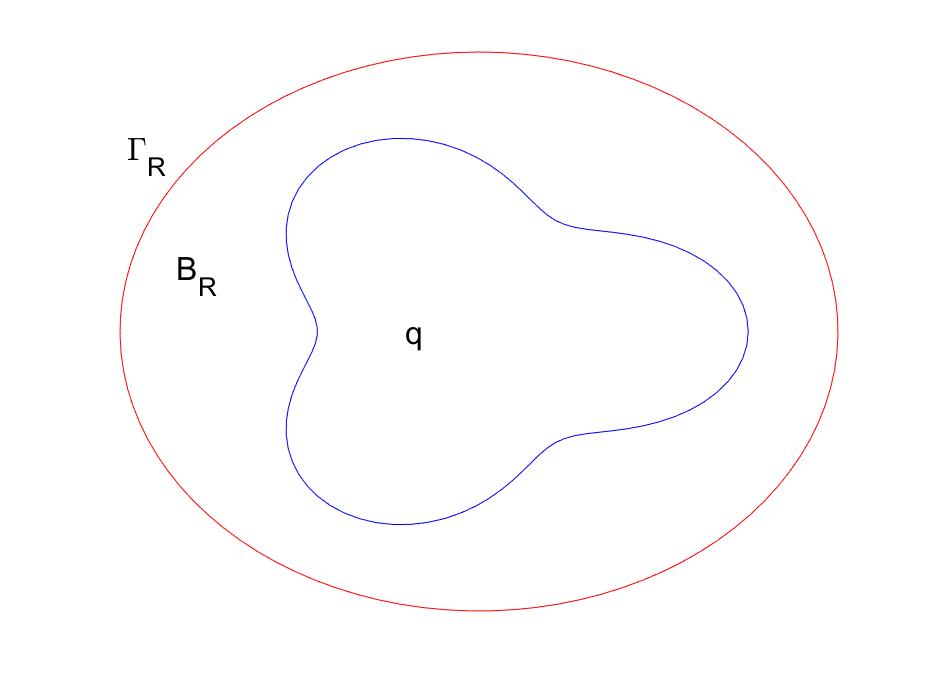}\\
  \caption{The geometry setting for the scattering problem}\label{f1}
\end{figure}
\subsection{Equivalent Formulation}

In the following, let $L^{2}(B_{R})$ be the usual Hilbert space of all square integrable functions, and let $H^{1}(B_{R})$ be the Sobolev space equipped with the norm
$$\|u\|_{H^{1}(B_{R})}=(\|u
\|^{2}_{L^{2}(B_{R})}+\|\nabla u\|^{2}_{L^{2}(B_{R})})^{\frac{1}{2}}.$$
Define the trace space $H^{s}(\Gamma_{R})$, $s\in \mathbb{R}$, as
$$H^{s}(\Gamma_{R})=\{u\in L^{2}(\Gamma_{R})\,|\,\|u\|_{H^{s}(\Gamma_{R})}<\infty\}$$
equipped with the norm
$$\|u\|_{H^{s}(\Gamma_{R})}=(\frac{|a_{0}|^{2}}{2}+\sum_{n\in\mathbb{Z}}(1+n^{2})^{s}(|a_{n}|^{2}+|b_{n}|^{2}))^{\frac{1}{2}},$$
where $a_{n}$ and $b_{n}$ are Fourier coefficients of $u\in H^{s}(\Gamma_{R}) $.

In the domain $\mathbb{R}^{2}\backslash \overline{B}_{R}$, the solution of equation (\ref{eq3}) has the form in the polar coordinates as follows \cite{Colton2019}:
\begin{equation}\label{eq4}
  u^{\rm s}(r,\theta)=\sum_{n\in \mathbb{Z}}\frac{H_{n}^{(1)}(kr)}{H_{n}^{(1)}(kR)}\hat{u}_{n}^{\rm s}e^{in\theta},
\end{equation}
where
\begin{equation*}
\hat{u}_{n}^{\rm s}=(2\pi)^{-1}\int_{0}^{2\pi}u^{\rm s}(R,\theta)e^{-in\theta}d\theta,
\end{equation*}
and $H_{n}^{(1)}$ is the Hankel function of the first kind with order $n$.

On the artificial boundary $\Gamma_{R}$, we can define the Dirichlet-to-Neumann (DtN) operator $T: H^{1/2}(\Gamma_{R})\rightarrow  H^{-1/2}(\Gamma_{R})$ as follows: for any $u^{\rm s}\in  H^{1/2}(\Gamma_{R})$,
\begin{equation}\label{eq4u}
  \frac{\partial u^{\rm s}}{\partial \nu}\mid_{\Gamma_{R}}=Tu^{\rm s}:=k\sum_{n\in \mathbb{Z}}\frac{H_{n}^{(1)'}(kR)}{ H_{n}^{(1)}(kR)}\hat{u}^{\rm s}_{n}e^{in\theta},
\end{equation}
where $\nu$ is the unit outward normal to $\Gamma_{R}$ (\cite{Bao2015}). Alternatively, the DtN operator $T$ can be expressed as
\begin{equation}\label{eq5}
\frac{\partial u^{\rm s}}{\partial \nu}\mid_{\Gamma_{R}}=Tu^{\rm s}:=\sum_{n=0}^{\infty}\frac{kH_{n}^{(1)'}(kR)}{2\pi H_{n}^{(1)}(kR)}\int_{0}^{2\pi}u^{\rm s}(R,\varphi)\cos(n(\theta-\varphi))d\varphi.
\end{equation}

The original scattering problem (\ref{EQ_1a})-(\ref{EQ_1c}) defined on $\mathbb{R}^{2}$ can be equivalently reduced to the following problem defined on the bounded domain \cite{Bao3},
\begin{subequations}\label{EQ_2}
\begin{numcases}{}
 \Delta u^{\rm s}+k^{2}(1+q(x))u^{\rm s} =-k^{2}q(x)u^{\rm i} \quad \text{in}\quad B_{R},   \\
 \frac{\partial u^{\rm s}}{\partial \nu}=Tu^{\rm s} \quad \text{on}\quad \Gamma_{R}.
\end{numcases}
\end{subequations}
Then, we have the weak formulation of the boundary value problem (\ref{EQ_2}): find $u^{\rm s}\in H^{1}(B_{R})$ such that
\begin{equation}\label{eq8}
  A(u^{\rm s},v)=\ell(v)\quad \forall v\in H^{1}(B_{R}),
\end{equation}
where the bilinear form $A(\cdot,\cdot): H^{1}(B_{R})\times H^{1}(B_{R})\to \mathbb{C}$ is defined by
\begin{equation*}
  A(u^{\rm s},v)=\int_{B_{R}}\nabla u^{\rm s}\cdot \nabla \overline{v}dx-k^{2}\int_{B_{R}}(1+q(x)u^{\rm s}\overline{v}dx-\int_{\Gamma_{R}}Tu^{\rm s}\overline{v}dS,
\end{equation*}
and the linear functional $\ell(\cdot): H^{1}(B_{R})\to\mathbb{C}$ is defined by
\begin{equation*}
\ell(v)=k^{2}\int_{B_{R}}q(x)u^{\rm i}\overline{v}dx.
\end{equation*}

Finally, we point out that given the incident field $u^{\rm i}$ and the scatterer $q(x)$, the direct scattering problem is to determine the scattered field $u^{\rm s}$.
\section{Bayesian Level Set Inversion}
\label{sec:3}
In this section, we adopt level set functions to reformulate the inverse medium scattering problem as a shape reconstruction problem.
\subsection{The Inverse Problem}
\begin{definition}
The scatterer $q(x)\in L^{\infty}(B_{R})$ is said to be admissible if there exists a compact domain $D \subset\subset B_{R}$ such that
\begin{numcases}{q(x)=}
b, & \text{for} $x\in D$,\nonumber\\
0, &  \text{for} $x\in \mathbb{R}^{2}\backslash D$, \nonumber
\end{numcases}
where $b>0$ is a constant. The set of all such scatterers will be denoted by $\mathcal{A}(B_{R})$.
\end{definition}

For the problem (\ref{EQ_2}), we assume that $M$ different wavenumbers $k:=k_{m}, m=1,\cdots, M$, are given. For each of these wavenumbers, there correspond to the incident waves $u_{mj}^{\rm i}=e^{ik_{m}x\cdot \textbf{d}_{j}}, j=1,\cdots, J$. Thus, for a given wavenumber $k_{m}$ and a given incident wave  $u_{mj}^{\rm i}$, we define the forward operator $G_{mj}: X\rightarrow Y$ by $u_{mj}^{\rm s}=G_{mj}(q(x))$, where $q(x)\in \mathcal{A}(B_{R}):=X$, $u_{mj}^{\rm s}\in H^{1}(B_{R}):=Y$. On the other hand, we denote the observation data with noise by
\begin{equation}\label{3-1}
  y_{mj}=O_{mj}\circ G_{mj}(q(x))+\eta_{mj},
\end{equation}
where $O_{mj}: Y\to \mathbb{C}^{N}$ denotes the collection of $N$ linear functionals on $Y$, $y_{mj}\in \mathbb{C}^{N}$, and $\eta_{mj}\sim \mathcal{N}(0,\Sigma_{0})$ is the additive Gaussian noise with covariance matrix $\Sigma_{0}\in \mathbb{R}^{N\times N}$. Gathering all the observations, one can rewrite (\ref{3-1}) as
\begin{equation}\label{eq9}
  y=O\circ G(q(x))+ \eta,
\end{equation}
where $y:=(y_{11},\cdots,y_{MJ})^{T}\in \mathbb{C}^{MJN}:=\mathcal{Y}$, $O\circ G := (u_{11}^{\rm s}(x_{1}),\cdots, u_{MJ}^{\rm s} (x_{N}))^{T}\in \mathbb{C}^{MJN}$ denotes the noise-free data with observation points $\{x_{n}\}_{n=1}^{N}\subseteq \Gamma_{R}$, and $\eta :=(\eta_{11},\cdots, \eta_{MJ})\sim \mathcal{N}(0,\Sigma)$ with covariance matrix $\Sigma=diag(\Sigma_{0},\cdots, \Sigma_{0})\in \mathbb{R}^{MJN\times MJN}$.

\subsection{Level Set Parameterization}
The scatterer $q(x)$ is characterized by
\begin{equation}\label{eq10}
  q(x)=\sum_{i=1}^{L}b_{i}\mathbb I_{B_{i}}(x),
\end{equation}
where $\{B_{i}\}_{i=1}^{L}$ are $L$ subdomains such that $B_{i}\cap B_{j}=\emptyset,\forall i\neq j$ and $\cup_{i=1}^{L}\overline{B}_{i}=\overline{B}_{R}$, $\mathbb I$ denotes the indicator function of a set, and the $\{b_{i}\}_{i=1}^{L}$ are known constants, $b_{i}\in \{b,0\}$. In this setting, the unknown scatterer would be determined by the domains $B_{i}, i=1, \cdots, L$. It is natural to make use of the level set representation of the domains through a continuous real-valued function $\phi: B_{R}\rightarrow \mathbb{R}$. To this purpose, we define $B_{i}\subseteq B_{R}$ by the level set function $\phi$,
\begin{equation}\label{eq11}
  B_{i}=\{x\in B_{R}|c_{i-1}\leq \phi(x)<c_{i}\},\quad i=1,\cdots L,
\end{equation}
where $c_{i}$ are constants with $-\infty=c_{0}<c_{1}<\cdots<c_{L}=\infty$, $i\in \mathbb{N}$. We define the level sets as
\begin{equation}\label{EQ}
  B_{i}^{0}=(\bigcup_{j=1}^{i}\overline{B}_{j})\cap \overline{B}_{i+1}=\{x\in B_{R}|\phi(x)=c_{i}\},\quad i=1,\cdots L-1.
\end{equation}
It is evident that the same domain $B_{i}$ can be represented by different level set functions $\phi_{1}$ and $\phi_{2}$, and however different domains can not be determined by the same level set representation. Therefore, we can use the level set representation to uniquely specify the domain $B_{i}$ by an associated level set function, $i=1, \cdots, L$. Let $\mathcal{X}=C(\overline{B}_{R},\mathbb{R})$, $F:\mathcal{X}\rightarrow X$ is the level set map described by
\begin{equation}\label{eq12}
  (F\phi)(x)\rightarrow q(x)=\sum_{i=1}^{L}b_{i}\mathbb I_{B_{i}}(x).
\end{equation}
Then we modify our operator $O\circ G$ into $\mathcal{G}=O\circ G\circ F: \mathcal{X}\rightarrow \mathcal{Y}$. As a result, the inverse problem can be reformulated as:  for given $y$, find $\phi$ such that
\begin{equation}\label{eq13s}
  y=\mathcal{G}(\phi)+\eta.
\end{equation}

\subsection{Bayesian Inference}
In the Bayesian framework, all quantities in (\ref{eq13s}) are viewed as random variables. Since it is assumed that $\eta\in \mathbb{R}^{MJN}$ is additive Gaussian, it is typically straightforward to write the likelihood function, i.e. the probability density of $y$ given $\phi$,
 \begin{equation}\label{15p}
  \pi(y|\phi)\propto\exp(-\frac{1}{2}|\mathcal{G}(\phi)-y|_{\Sigma}^{2}),
\end{equation}
where $|\cdot|_{\Sigma}:=|\Sigma^{-\frac{1}{2}}\cdot|$ denotes the weighted norm in terms of the Euclidean norm $|\cdot |$. In the following, we denote $\frac{1}{2}|\mathcal{G}(\phi)-y|_{\Sigma}^{2}$ by the potential $\Phi(\phi;y)$. We assume that the prior measure of the unknown $\phi$ is $\mu_{0}$, the posterior measure $\mu_{y}$ is represented as the Radon-Nikodym derivative
with respect to the prior measure $\mu_{0}$:
\begin{equation}\label{15pp}
  \frac{d\mu_{y}}{d\mu_{0}}(\phi)=\frac{1}{Z} \exp(-\Phi(\phi;y)).
\end{equation}
where $Z=\int_{\mathcal{X}} \pi(y|\phi)\mu_{0}(d\phi)$ is a normalization constant. The equation (\ref{15pp}) can be viewed as the Bayes' rule in the infinite-dimensional setting.
\subsubsection{Whittle-Mat\'{e}rn Gaussian Random Field Prior}
We now introduce Gaussian prior of Whittle-Mat\'{e}rn type with covariance \cite{Roininen}
\begin{equation}\label{eq15}
 c(x,y)=\sigma^{2}\frac{2^{2-\alpha}}{\Gamma(\alpha-1)}(\frac{|x-y|}{l})^{\alpha-1}K_{\alpha-1}(\frac{|x-y|}{l}), \quad x,y \in \mathbb{R}^{2},
\end{equation}
where $K_{\alpha-1}$ is the modified Bessel function of the second kind of order $\alpha-1$, $\sigma^{2}>0$ is the variance, $l>0$ is the characteristic length scale, and $\Gamma(\cdot)$ is the Gamma function. We generate the samples from the Whittle-Mat\'{e}rn prior by solving a stochastic partial differential equation
\begin{equation}\label{eq16}
 (I-l^{2}\Delta)^{\frac{\alpha}{2}}\phi=l\sqrt{\varsigma}\xi,
\end{equation}
where $(I-l^{2}\Delta)^{\frac{\alpha}{2}}$ is a pseudo-difference operator defined by its Fourier transform, $\xi$ is a Gaussian white noise, and $\varsigma=\sigma^{2}\frac{4\pi\Gamma(\alpha)}{\Gamma(\alpha-1)}$ is a constant.
Set $\tau=l^{-1}>0$, we have the stochastic partial differential equation
\begin{equation}\label{eq17}
  \mathcal{C}_{\alpha,\tau}^{-\frac{1}{2}}\phi=\xi,
\end{equation}
where $\mathcal{C}_{\alpha,\tau}=\varsigma\tau^{2\alpha-2}(\tau^{2}I-\Delta)^{-\alpha}$ denotes the covariance operator of prior distribution $\mu_{0}$, $\alpha$ controls the regularity of the samples, and $\tau$ represents the inverse length scale of the samples. In what follows, assume that $\mathcal{A}:= \Delta$ is Laplacian with Neumann boundary conditions on $B_{R}$, and its domain is given by
\begin{equation}\label{eq19}
  B_{R}(\mathcal{A}):=\{\phi:B_{R}\rightarrow \mathbb{R}\,|\,\phi\in H^{2}(B_{R};\mathbb{R}), \frac{\partial \phi}{\partial \nu}=0\;\text{on}\; \partial B_{R}\}.
\end{equation}
In Fig. \ref{prior1} and Fig. \ref{prior2}, we display random samples obtained from (\ref{eq17}) regarding to different values of the inverse length scale $\tau$ and the regularity $\alpha$. These samples are constructed in the domain $B_{R}$ with $R=1$.
\begin{figure}
  \centering
  \begin{subfigure}
  \centering
  \includegraphics[width=1.6in,height=1.4in]{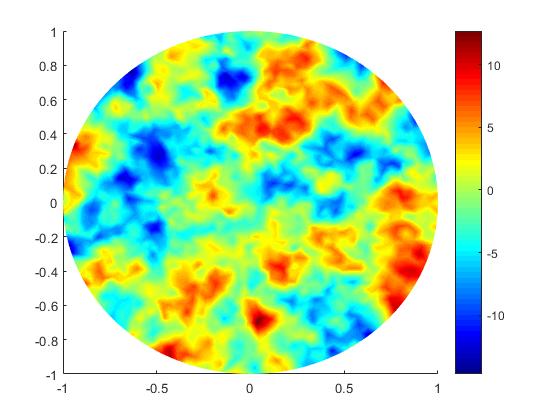}
  \end{subfigure}
  \begin{subfigure}
  \centering
  \includegraphics[width=1.6in,height=1.4in]{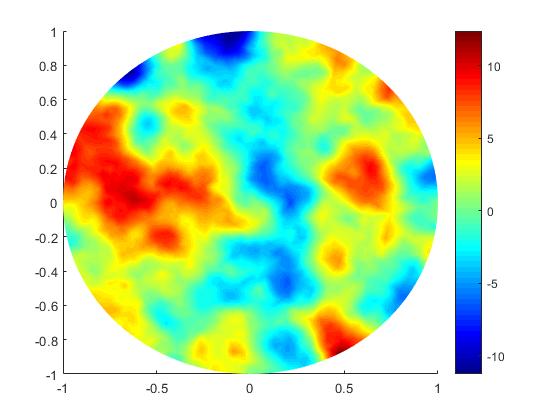}
  \end{subfigure}
  \begin{subfigure}
  \centering
  \includegraphics[width=1.6in,height=1.4in]{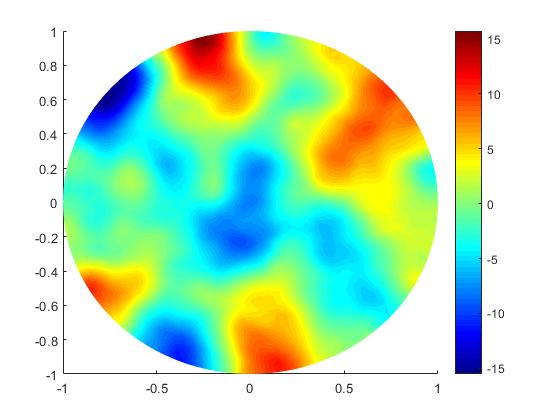}\\
  \end{subfigure}
  \caption{Samples from the prior with $\alpha=2,3,4$, for $\tau=10$.}
  \label{prior1}
\end{figure}

\begin{figure}
  \centering
  \begin{subfigure}
  \centering
  \includegraphics[width=1.6in,height=1.4in]{prior_3_01.jpg}
  \end{subfigure}
  \begin{subfigure}
  \centering
  \includegraphics[width=1.6in,height=1.4in]{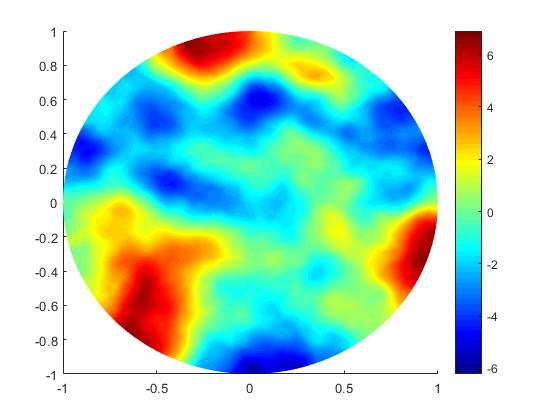}
  \end{subfigure}
  \begin{subfigure}
  \centering
  \includegraphics[width=1.6in,height=1.4in]{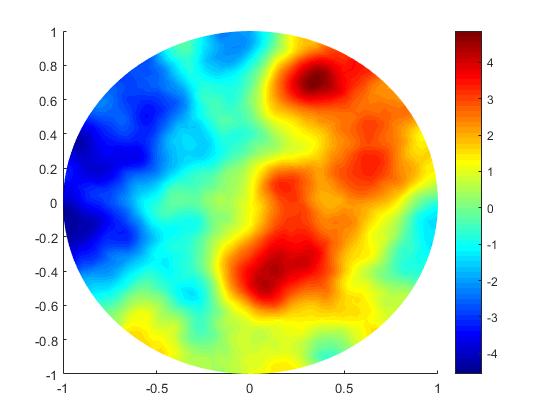}\\
  \end{subfigure}
  \caption{Samples from the prior with inverse length scale $\tau=10,\frac{20}{3},5$, for $\alpha=3$.}
  \label{prior2}
\end{figure}
\subsubsection{Well-Posedness of the Posterior Distribution}
We now discuss the well-posedness of the posterior distribution for the IMSP. It is clear that the level set map is discontinuous, and due to this fact, we get that the map $\mathcal{G}$ is discontinuous. Although the well-posedness of Bayesian inversion relies on the continuity of the map $\mathcal{G}$, it demonstrates (\cite{Iglesias}) that the discontinuity set is a probability zero event under the Gaussian prior. As a result, $F$ will be almost surely continuous under the prior, and it will be given in the following Lemma \ref{prop2}. Thus, we are able to get the measurability required in the Bayes' theorem \cite{Stuart}. Furthermore, we need to verify some regularity properties of the potential $\Phi(\phi; y)$ which satisfy the assumptions of the Bayes' theorem. Prior to the presentation, we define a complete probability space $(\mathcal{X}, \Xi, \mu_{0})$ for the unknown $\phi$, where $\mathcal{X}$ denotes a separable Banach space, and $\Xi$ is the $\sigma$-algebra.

\begin{lemma}\label{prop2}
Define the map $F:\mathcal{X} \rightarrow X$ given by $(\ref{eq12})$. Let $\phi\in \mathcal{X}$ be such that $m(B_{i}^{0})=0$, for $i=1,\cdots,L-1$. Assume that $\{\phi_{\epsilon}\}_{\epsilon>0}\subseteq C(\overline{B}_{R})$ denotes an approximate sequence of level set functions such that $\|\phi_{\epsilon}-\phi\|_{\infty}\rightarrow 0$. Then $F(\phi_{\epsilon})\rightarrow F(\phi)$ in measure.
\end{lemma}
\begin{remark}
Here $m(B_{i}^{0})=0$ denotes the Lebesgue measure of the set $B_{i}^{0}$. The proof of this lemma is almost identical to that of Proposition 3.5 in \cite{Dunlop2}, and we omit the details here.
\end{remark}
\begin{proposition}\label{prop1}
The potential $\Phi(\phi; y)$  and probability measure $\mu_{0}$ on the measure space $(\mathcal{X},\Xi)$ satisfy the following properties:\\
(1) for every $r>0$, there is a $K_{1}=K_{1}(r)$ such that, for all $\phi\in \mathcal{X}$ and $y\in \mathcal{Y}$ with $|y|_{\Sigma}<r$,
\begin{equation}\label{eq20}
  0\leq \Phi(\phi;y)\leq K_{1};
\end{equation}
(2) for any fixed $y\in \mathcal{Y}$, $\Phi(\cdot;y): \mathcal{X}\rightarrow\mathbb{R}$, is continuous $\mu_{0}$-almost surely on the probability space $(\mathcal{X},\Xi,\mu_{0})$;\\
(3) for every $r>0$, there exists a $K_{2}=K_{2}(r)$ such that, for all $\phi\in \mathcal{X}$, and $y_{1},y_{2}\in \mathcal{Y}$ with $\max\{|y_{1}|_{\Sigma}, |y_{2}|_{\Sigma}\}<r$,
\begin{equation}\label{eq21}
  |\Phi(\phi;y_{1})-\Phi(\phi;y_{2})|\leq K_{2}|y_{1}-y_{2}|_{\Sigma}.
\end{equation}
\end{proposition}
\begin{proof}

(1) From the problem (\ref{EQ_2}), it can be observed that the map $G$ is nonlinear with respect to $q$. We know that the following estimate holds \cite{Bao3}
\begin{equation}\label{eqt1}
 \|u^{\rm s}\|_{H^{1}(B_{R})}= \|G(q)\|_{H^{1}(B_{R})}\leq C\|q\|_{L^{\infty}(B_{R})}\|u^{\rm i}\|_{L^{2}(B_{R})}.
\end{equation}
$O: Y\rightarrow \mathcal{Y}$ is the bounded linear map, and $\|F(\phi)\|_{\infty}$ is bounded uniformly over $\phi\in \mathcal{X}$. Hence
\begin{equation}\label{eqt4}
  |\mathcal{G}(\phi)|_{\Sigma}=|O\circ G\circ F(\phi)|_{\Sigma}\leq C.
\end{equation}
Note that
\begin{equation}\label{eqt2}
  \Phi(\phi;y)=\frac{1}{2}|y-\mathcal{G}(\phi)|_{\Sigma}^{2}\leq|y|^{2}_{\Sigma}+|\mathcal{G}(\phi)|^{2}_{\Sigma}.
\end{equation}
Then, for any $y\in \mathcal{Y}$ with $|y|_{\Sigma}<r$, we obtain the bound
\begin{equation}\label{eqt3}
  \Phi(\phi;y)\leq C(1+r^{2})=: K_{1}.
\end{equation}

(2) It is known that the map $G:X\rightarrow Y$ is continuous \cite{Bao3}, i.e.
\begin{equation*}
  \|G(q_{1})-G(q_{2})\|_{H^{1}(B_{R})}\leq C\|q_{1}-q_{2}\|_{L^{\infty}(B_{R})}\|u^{\rm i}\|_{L^{2}(B_{R})},
\end{equation*}
and $O: Y\rightarrow \mathcal{Y}$ is the linear continuous map. Therefore, the discontinuity set of $\mathcal{G}$ is determined by the discontinuity set of the level set map $F$. However, since we assume that $\phi\sim \mu_{0}$ is a Gaussian measure, it follows from the Proposition 2.8 in \cite{Iglesias} that $m(B_{i}^{0})=0$, $\mu_{0}$-almost surly, $i=1,\cdots,L-1$. In brief, the level sets of the Gaussian random field have Lebesgue measure zero. By Lemma \ref{prop2}, we can obtain that $\|\phi_{\epsilon}-\phi\|_{\infty}\rightarrow 0$ implies that $F(\phi_{\epsilon})\rightarrow F(\phi)$ in measure. Therefore, $\Phi(\cdot;y)$ is continuous $\mu_{0}$-almost surely .

(3) Let $\phi\in \mathcal{X}$ and $y_{1}, y_{2} \in \mathcal{Y}$  with $\max\{|y_{1}|_{\Sigma}, |y_{2}|_{\Sigma}\}<r$. It follows that
\begin{align}\label{eq22}
   |\Phi(\phi;y_{1})-\Phi(\phi;y_{2})| & = |\frac{1}{2}|y_{1}-\mathcal{G}(\phi)|_{\Sigma}^{2}-\frac{1}{2}|y_{2}-\mathcal{G}(\phi)|_{\Sigma}^{2}| \nonumber \\
                                 & =\frac{1}{2}|\langle y_{1}-y_{2}, y_{1}+y_{2}-2\mathcal{G}(\phi)\rangle_{\Sigma}| \nonumber\\
                                 & \leq (|y_{1}|_{\Sigma}+|y_{2}|_{\Sigma}+2|\mathcal{G}(\phi)|_{\Sigma})|y_{1}-y_{2}|_{\Sigma} \nonumber\\
                                 & \leq (r+2|\mathcal{G}(\phi)|_{\Sigma})|y_{1}-y_{2}|_{\Sigma}  \nonumber\\
                                 & =: K_{2}|y_{1}-y_{2}|_{\Sigma}.\nonumber
\end{align}
\end{proof}
\begin{definition}
Let $\nu_{0}$ be a common reference measure. The Hellinger distance between $\mu$ and $\mu'$ with common reference measure $\nu_{0}$ is
\begin{equation}\label{eq25}
  d_{Hell}(\mu,\mu')=\sqrt{\frac{1}{2}\int(\sqrt{\frac{d\mu}{d\nu_{0}}}-\sqrt{\frac{d\mu'}{d\nu_{0}}})^{2}d\nu_{0}}.
\end{equation}
\end{definition}
\begin{theorem}\label{th2}
Assume that $\phi\sim \mu_{0}:=\mathcal{N}(0,\mathcal{C}_{\alpha,\tau})$. Then the following results hold:\\
(1) The posterior measure $\mu_{y}$ exists and is absolutely continuous with respect to $\mu_{0}$, i.e. $\mu_{y}\ll \mu_{0}$, with
Radon-Nikodym derivative given by (\ref{15pp}).\\
(2) $\mu_{y}$ is locally Lipschitz in the data $y$, with respect to the Hellinger distance: if $\mu_{y}$ and $\mu_{y'}$ are two measures with data $y$ and $y'$, then
for all $y$ and $y'$ with $\max\{|y|_{\Sigma},|y'|_{\Sigma}\}<r$, there exists a constant $C=C(r)$ such that
\begin{equation}\label{eq24}
 d_{Hell}(\mu_{y},\mu_{y'})\leq C|y-y'|_{\Sigma}.
\end{equation}
\end{theorem}
\begin{proof}
From the Proposition \ref{prop1} (2), we get that $\Phi(\cdot;y)$ is continuous $\mu_{0}$-almost surely. Using the $\mu_{0}$-almost surely continuity, it establishes the measurability in Lemma 6.1 (\cite{Iglesias}). Then, the first result follows from the Theorem 6.29 in \cite{Stuart}. For the Lipschitz continuity of the $\mu_{y}$, it could be proved by the Theorem 4.5 in \cite{Dashti}. Therefore, we omit the details here.
\end{proof}

\section{Numerical Experiments}
\label{sec:4}
In this section, some numerical results are presented to demonstrate performance of the proposed method. In particular, we compare the results of the Bayesian level set method with those of the regular Bayesian approach.
\subsection{Sampling Algorithm}
The Markov chain Monte Carlo (MCMC) methods are usually applied to draw the samples from the posterior distribution $\mu_{y}$ defined above. In this work, we employ the preconditioned Crank-Nicolson (pCN) algorithm \cite{Cotter}, which is described in Algorithm \ref{alg:A}. We adopt the proposal variance parameter $\beta\in (0, 1]$ to control the stepsize in numerical implementations. We take the choice of $\beta=0.007$ for a compromise between the acceptance rate and the exploration of the state space. The proposed pCN-MCMC algorithm is performed with samples $N_{s}=10^{4}$. We take the last $2\times10^{3}$ samples to compute the conditional mean (CM) estimates.
\begin{algorithm}[htb]
  \caption{The pCN-MCMC algorithm.}
  \label{alg:A}
  \begin{algorithmic}[1]
  \State Collect the scattered field measured data over all frequencies $k_{m}$, $m=1,\cdots, M$ and the incident direction $\textbf{d}_{j}$, $j=1,\cdots J$.
  \State Set $s=0$. Choose an initial state $\phi^{(0)}\in \mathcal{X}$.
   \For{$s=0$ to $N_{s}$}
      \State Propose $\psi^{(s)}=\sqrt{1-\beta^{2}}\phi^{(s)}+\beta\xi^{(s)}$,\quad $\xi^{(s)}\sim \mathcal{N}(0,\mathcal{C}_{\alpha,\tau})$;
      \State Draw $\theta\sim U[0,1]$
      \State Let $a(\phi^{(s)},\psi^{(s)}):=\min\{1,\exp(\Phi(\phi^{(s)})-\Phi(\psi^{(s)}))\}$;
      \If {$\theta\leq a$}
          \State $\phi^{(s+1)}= \psi^{(s)}$;
      \Else
          \State $\phi^{(s+1)} = \phi^{(s)}$;
      \EndIf
  \EndFor
  \end{algorithmic}
\end{algorithm}
\subsection{Data and Parameters}
 We consider the case of $R=1$ and discretize the domain with $16512$ elements uniformly with meshsize $h=2.45\times 10^{-2}$.  Meanwhile, we adopt a uniform mesh with meshsize $\widehat{h}=2h$ for the application of the pCN-MCMC algorithm. The synthetic data is generated by solving the forward model with the noise being added, and the data are measured on the boundary $\Gamma_{R}$. We assume that the noise is Gaussian, $\eta \sim \mathcal{N}(0,\gamma^{2}I)$, where $\gamma=0.005$. The number of incident directions $\textbf{d}_{j}$ is taken as $J=5$, and the incident directions $\textbf{d}_{j}$ are equally distributed around $\Gamma_{R}$ with $\textbf{d}_{1}=(1,0)$. The wavenumber varies from $k=0.5\pi$ to $k=2.5\pi$ with $M=6$.
\subsection{Numerical Results}
We test the regular Bayesian approach and the Bayesian level set method on the following examples. In both methods, the prior is taken to be a zero mean Gaussian with Mat\'{e}rn covariance, i.e. $\mathcal{N}(0,\mathcal{C}_{\alpha,\tau})$.

\textbf{Example 1.} The true scatterer has the form of
 \begin{equation}
q^{\dag}(x)=\left\{
             \begin{array}{lr}
             1, \quad x\in D  ,&\\
             0, \quad x\in B_{R}\backslash D  ,&
             \end{array}
\right.
\end{equation}
where $D=\{(x,y)\in \mathbb{R}^{2}:x^2+(y-\sqrt[3]{x^{2}})^2\leq 1\}$ is a love-shaped scatterer, as shown in Fig. \ref{fig5}. In the level set context, we parameterize $D$ in terms of the level set function given by $D=\{x\in D|\phi(x)\geq0\}$, and the corresponding level set map is $F(\phi)=\mathbb I_{D}$. Fixing $\alpha=3$, we apply the regular Bayesian approach and the Bayesian level set method to recover the shape of the scatterer with different inverse length scales, respectively. The reconstructed results are presented in Fig. \ref{fig6}. One can see from Fig. \ref{fig6} that both methods are effective to recover the shape of the scatterer. However, compared to the regular Bayesian approach, the Bayesian level set method shows the advantage of identifying the sharp boundary of the scatterer. One of the reasons is that we are able to make good use of the information on $q(x)$ under the framework of Bayesian level set method.

\begin{figure}[htp]
  \centering
  \begin{subfigure}
  \centering
  \includegraphics[width=1.6in,height=1.4in]{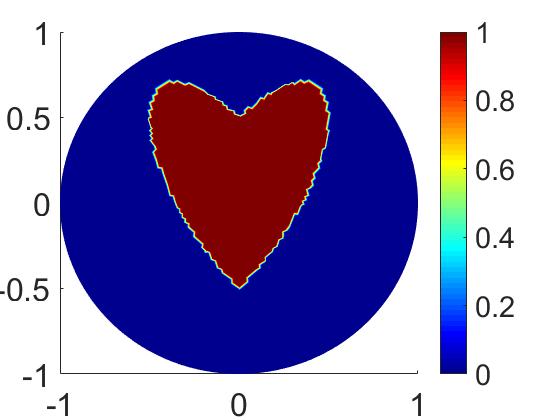}
  \end{subfigure}
 \caption{The true scatterer $q^{\dagger}$}
  \label{fig5}
\end{figure}
\begin{figure}[htp]
  \centering
  \begin{subfigure}
  \centering
  \includegraphics[width=1.6in,height=1.4in]{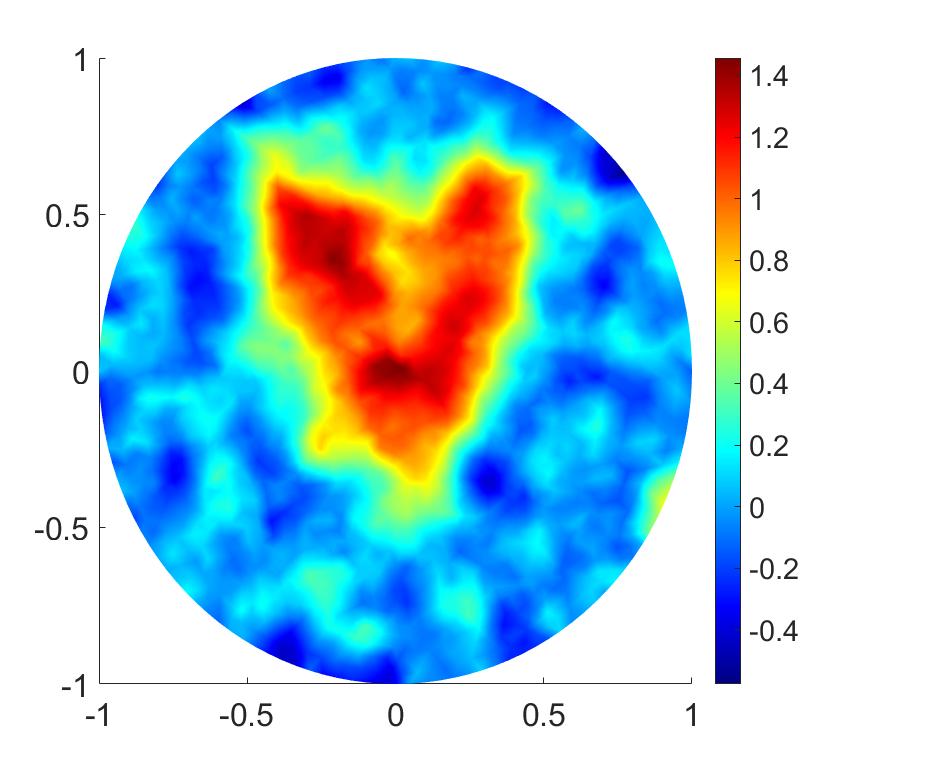}
  \end{subfigure}
  \begin{subfigure}
  \centering
  \includegraphics[width=1.6in,height=1.4in]{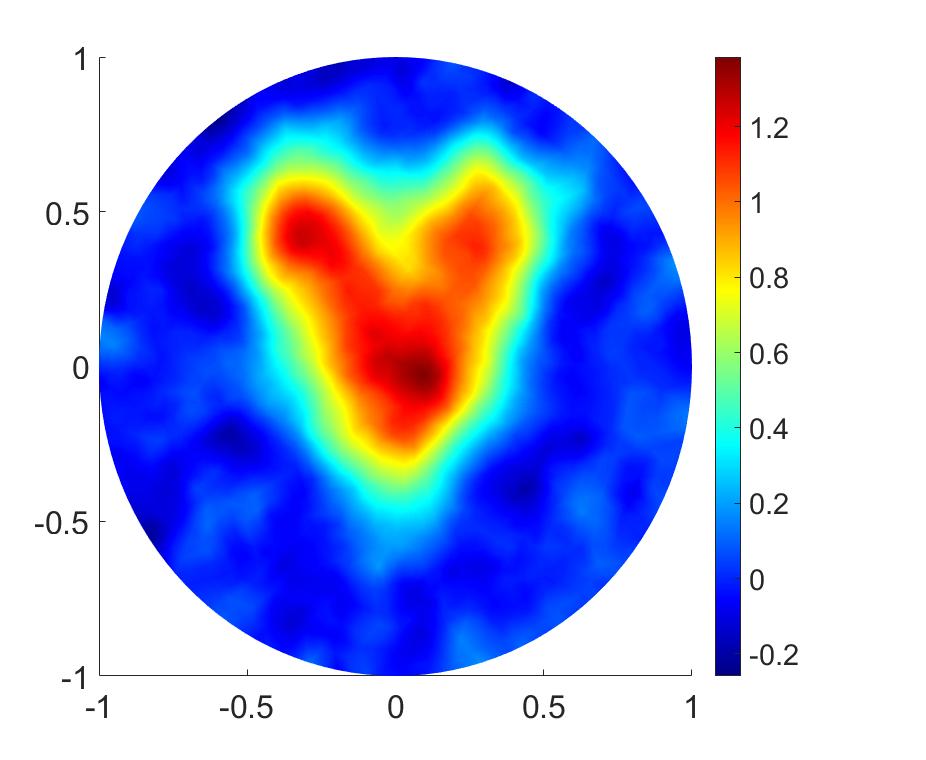}
  \end{subfigure}
  \begin{subfigure}
  \centering
  \includegraphics[width=1.6in,height=1.4in]{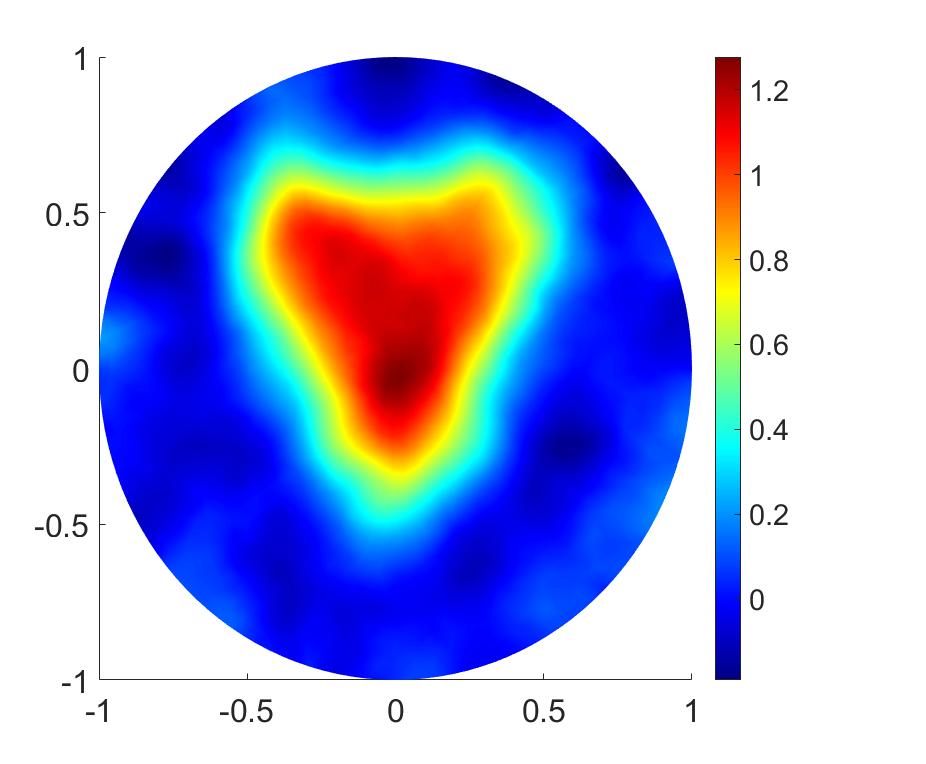}\\
  \end{subfigure}
  \begin{subfigure}
  \centering
  \includegraphics[width=1.6in,height=1.4in]{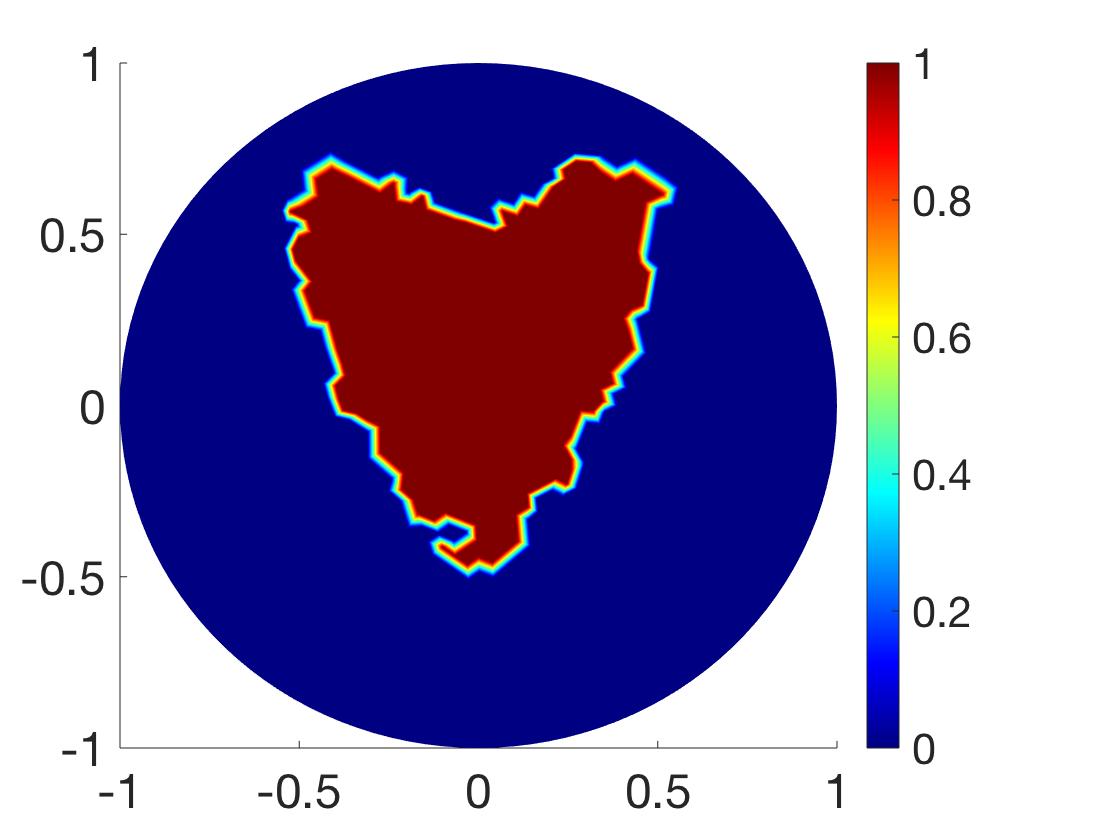}
  \end{subfigure}
  \begin{subfigure}
  \centering
  \includegraphics[width=1.6in,height=1.4in]{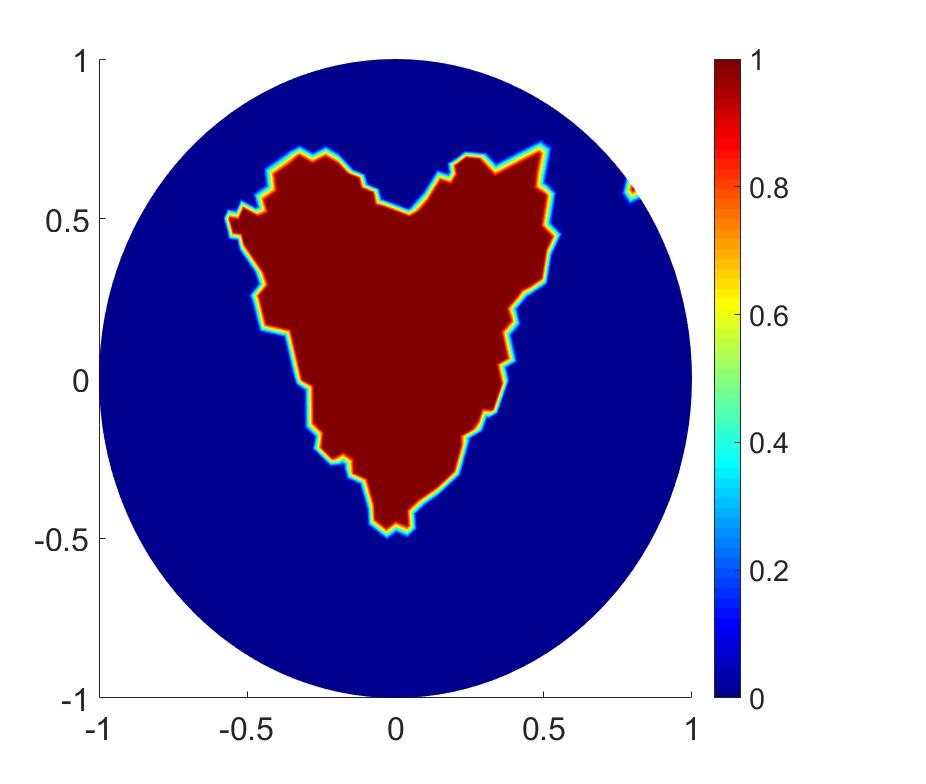}
  \end{subfigure}
  \begin{subfigure}
  \centering
  \includegraphics[width=1.6in,height=1.4in]{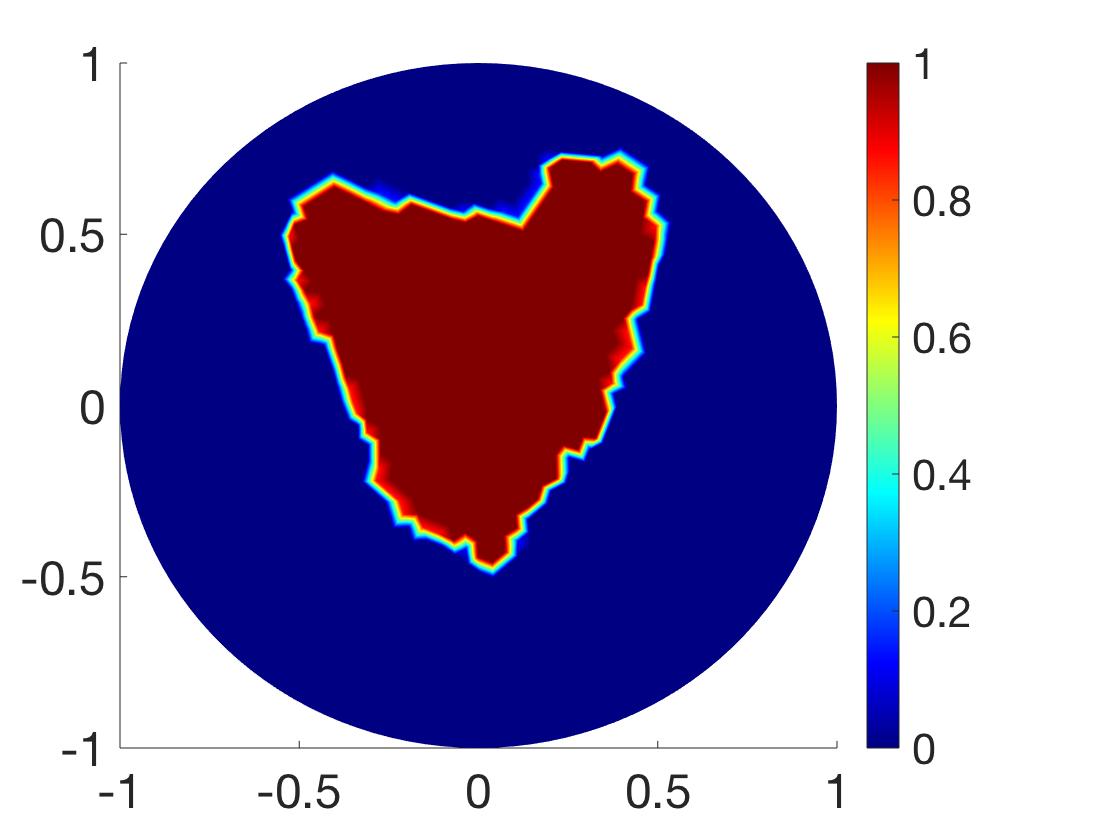}\\
  \end{subfigure}
 \caption{The reconstructions of a love-shaped scatterer for the regular Bayesian method (top block) and the Bayesian level set method (bottom block) with $\tau=10, \frac{20}{3}, 5$, $\alpha=3$.}\label{fig6}
\end{figure}

\textbf{Example 2.} The expression of the scatterer is the same as that of the Example 1, where D is a cross-shaped scatterer plotted in Fig. \ref{fig7}. We define the level set map $(F\phi)(x)=\mathbb I_{D}$ with $D=\{x\in D|\phi(x)\geq0\}$, and the zero level set presents the unknown boundary $\partial D$. Fixing the $\alpha=2$, we display the reconstructed results in Fig. \ref{fig8}. It can be observed that the Bayesian approach coupled with the level set method is a better alternative to recover the boundary of the scatterer.
\begin{figure}[htp]
  \centering
  \begin{subfigure}
  \centering
  \includegraphics[width=1.6in,height=1.4in]{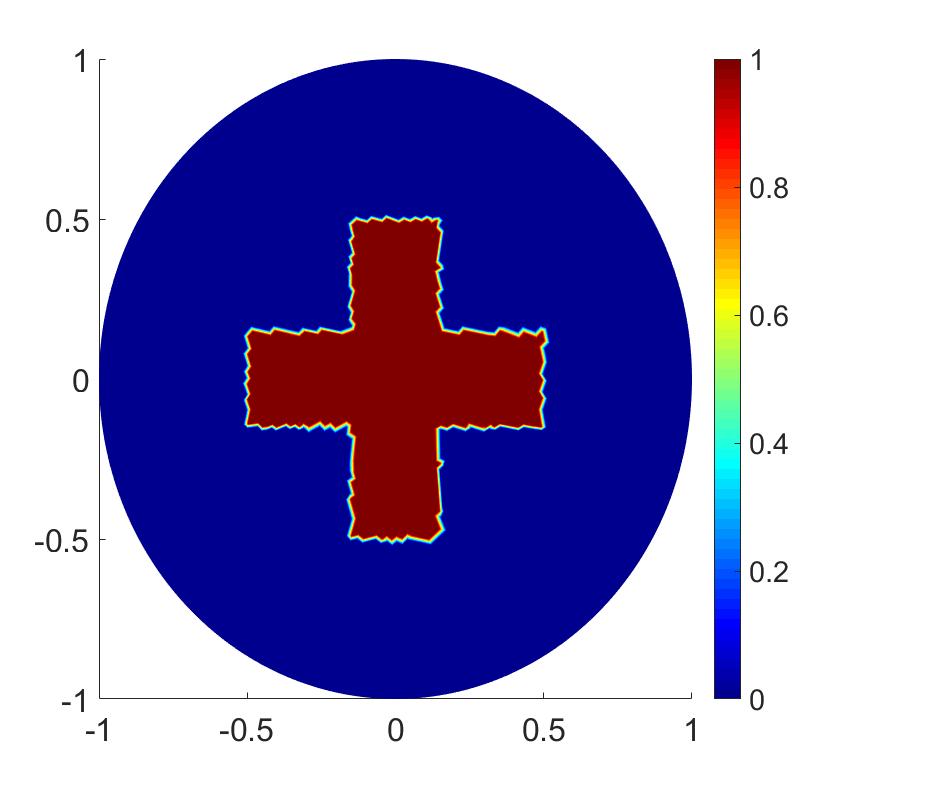}
  \end{subfigure}
 \caption{The true scatterer $q^{\dagger}$}
  \label{fig7}
\end{figure}

\begin{figure}[htp]
  \centering
  \begin{subfigure}
  \centering
  \includegraphics[width=1.6in,height=1.4in]{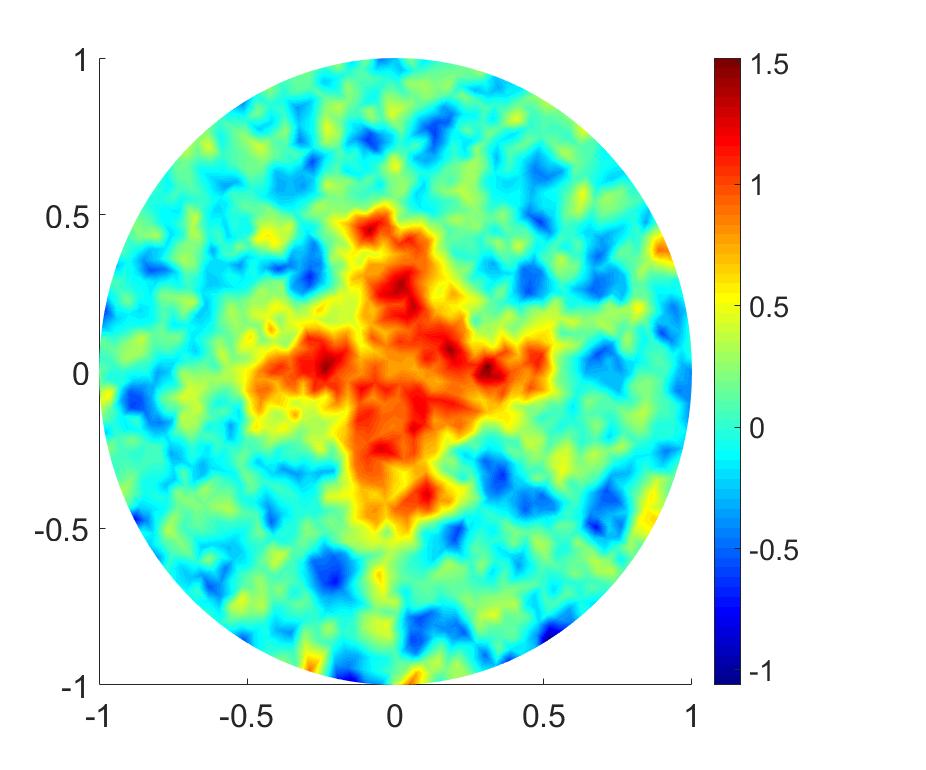}
  \end{subfigure}
  \begin{subfigure}
  \centering
  \includegraphics[width=1.6in,height=1.4in]{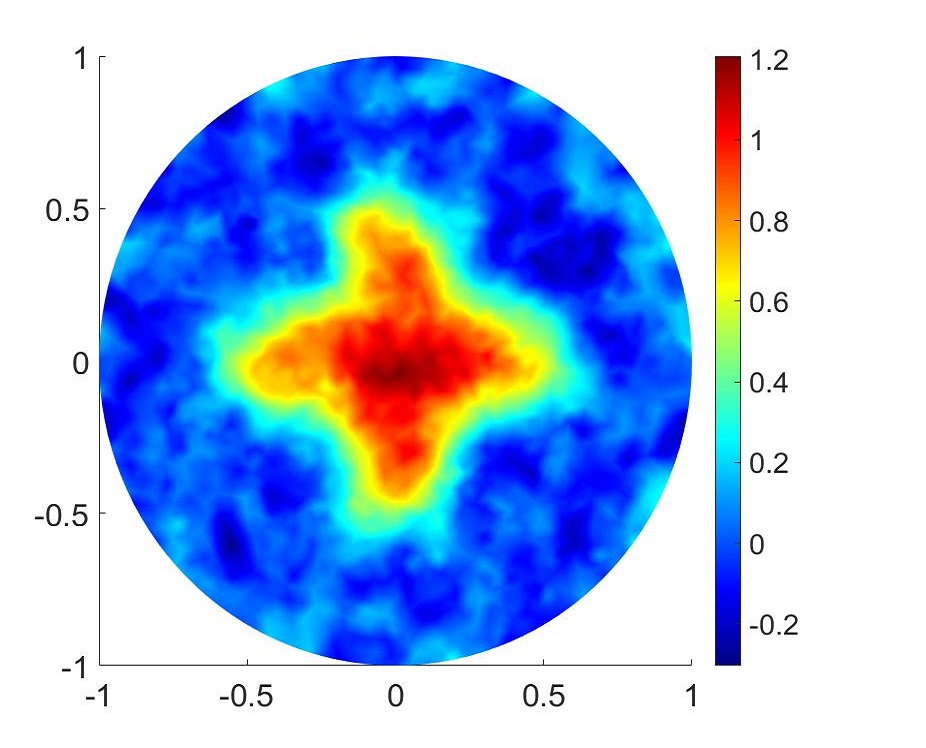}
  \end{subfigure}
  \begin{subfigure}
  \centering
  \includegraphics[width=1.6in,height=1.4in]{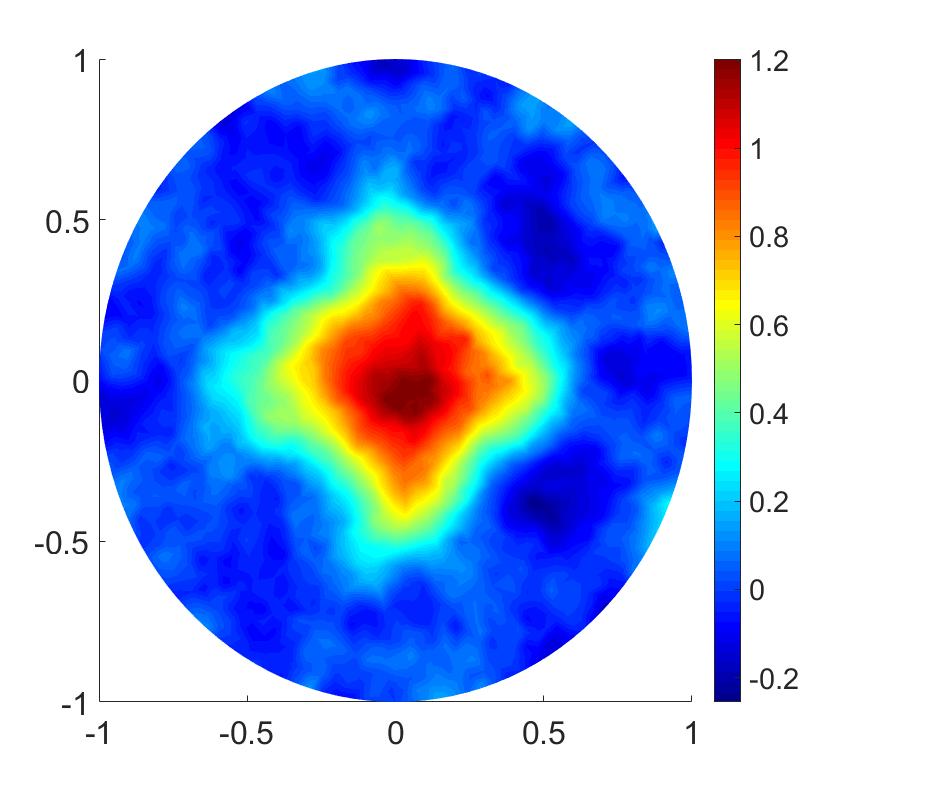}\\
  \end{subfigure}
  \begin{subfigure}
  \centering
  \includegraphics[width=1.6in,height=1.4in]{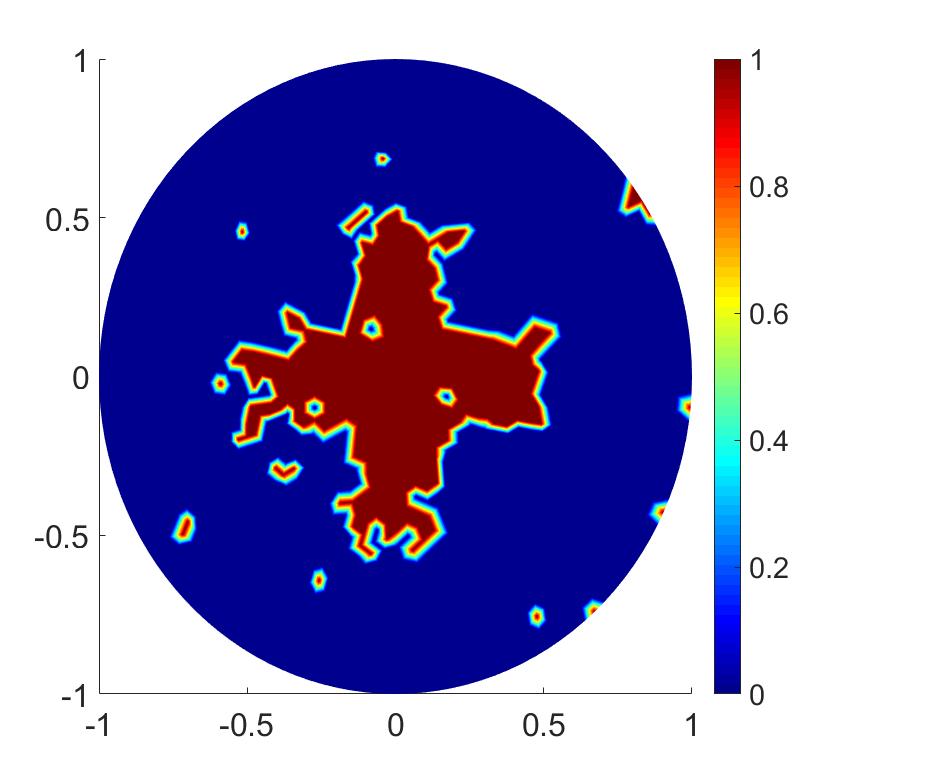}
  \end{subfigure}
  \begin{subfigure}
  \centering
  \includegraphics[width=1.6in,height=1.4in]{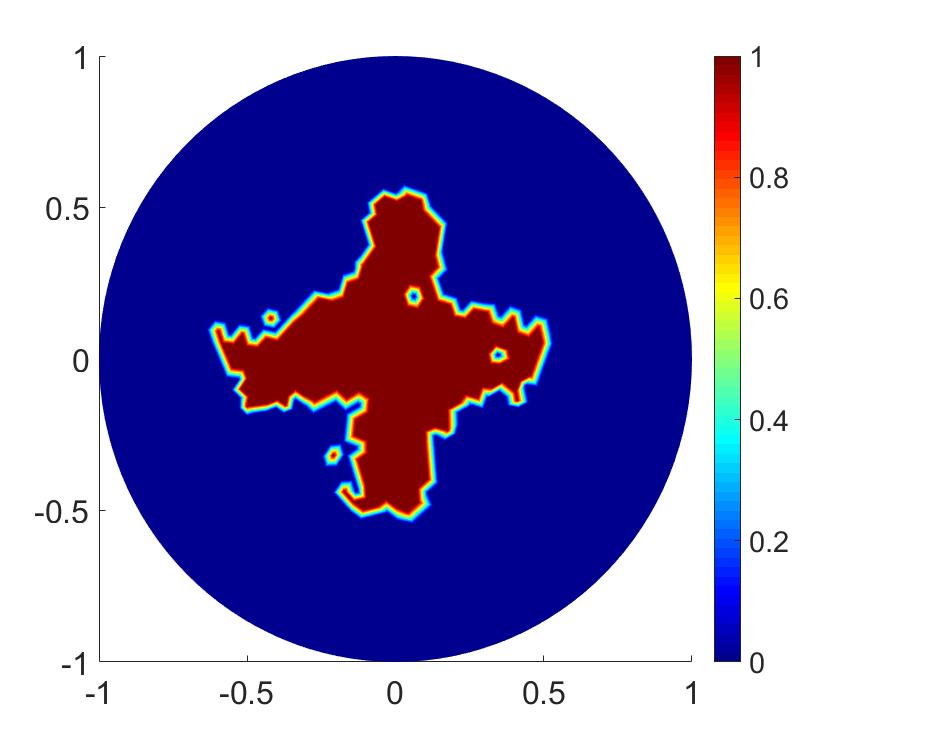}
  \end{subfigure}
   \begin{subfigure}
  \centering
  \includegraphics[width=1.6in,height=1.4in]{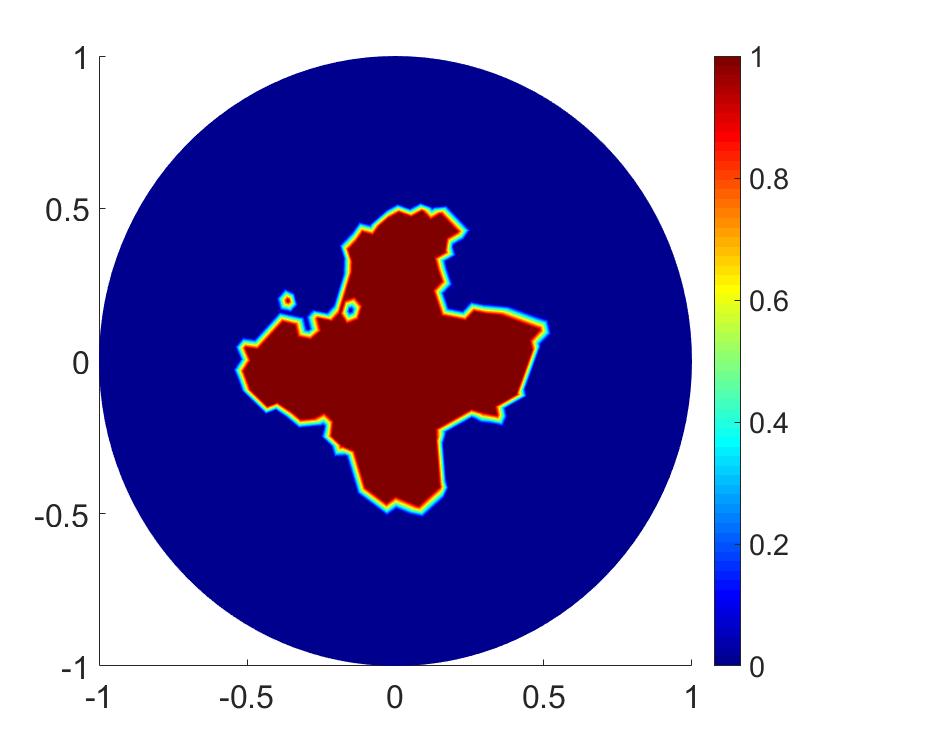}\\
  \end{subfigure}
 \caption{ The reconstructions of a cross-shaped scatterer for the regular Bayesian method (top block) and the Bayesian level set method (bottom block) with $\tau=10, 5, \frac{10}{3}$, $\alpha=2$.}\label{fig8}
\end{figure}

\textbf{Example 3.} The scatterers are two disjoint domains satisfying
\begin{equation}
q^{\dagger}(x)=\left\{
             \begin{array}{lr}
             3, \quad x\in D_{1}\, \text{or}\, D_{2}  &\\
             0, \quad x\in B_{R}\backslash (D_{1}\cup D_{2}),  &
             \end{array}
\right.
\end{equation}
where $D_{1}=\{(x,y)\in \mathbb{R}^{2}:(x+0.3)^2+(y-0.3)^2\leq0.04\}$ and $D_{2}=\{(x,y)\in \mathbb{R}^{2}:(x-0.3)^2+(y+0.3)^2\leq0.04\}$ are shown in Fig. \ref{fig9}. For the Bayeian level set method, $D$ is characterized by the level set function, and the corresponding level set map is $F(\phi)=3\mathbb I_{D_{1}}+3\mathbb I_{D_{2}}$ with $D_{i}=\{x\in D_{i}|\phi(x)\geq0\}, i=1, 2$. We take $\alpha=3$, and show the posterior samples with inverse length scale $\tau=10, \frac{20}{3}, 5$ in Fig. \ref{fig10}. As we can see from the results, the Bayesian level set method is more suitable in identifying the boundary of the scatterer.

\begin{figure}[htp]
  \centering
  \begin{subfigure}
  \centering
  \includegraphics[width=1.6in,height=1.4in]{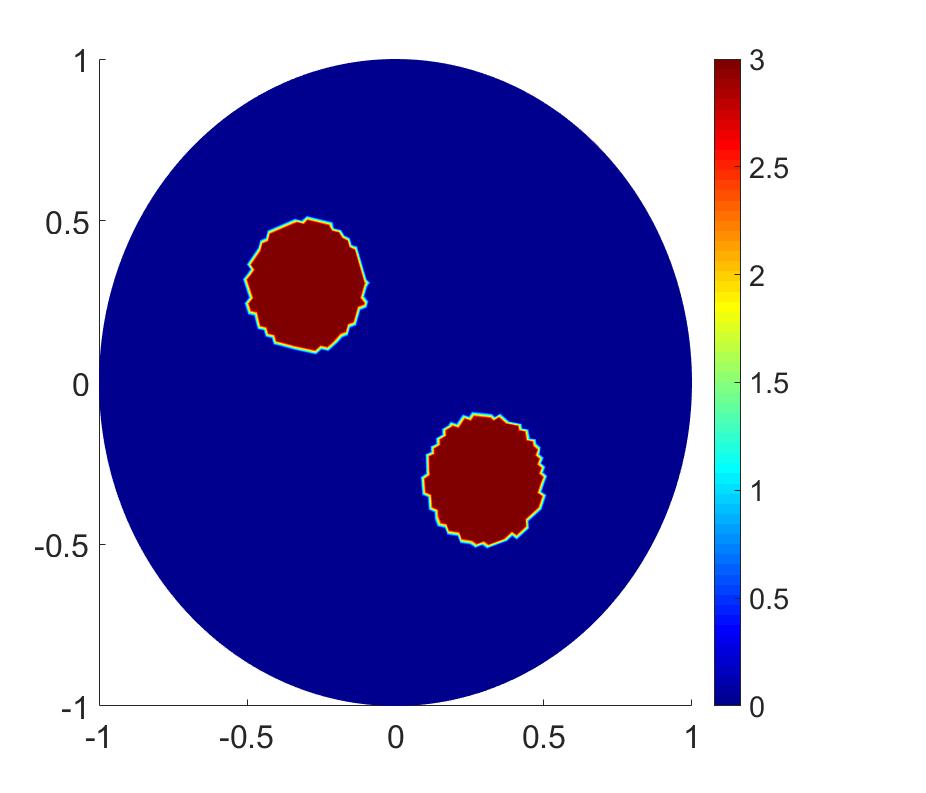}
  \end{subfigure}
 \caption{The true scatterer $q^{\dagger}$}
  \label{fig9}
\end{figure}

\begin{figure}[htp]
  \centering
   \begin{subfigure}
  \centering
  \includegraphics[width=1.6in,height=1.4in]{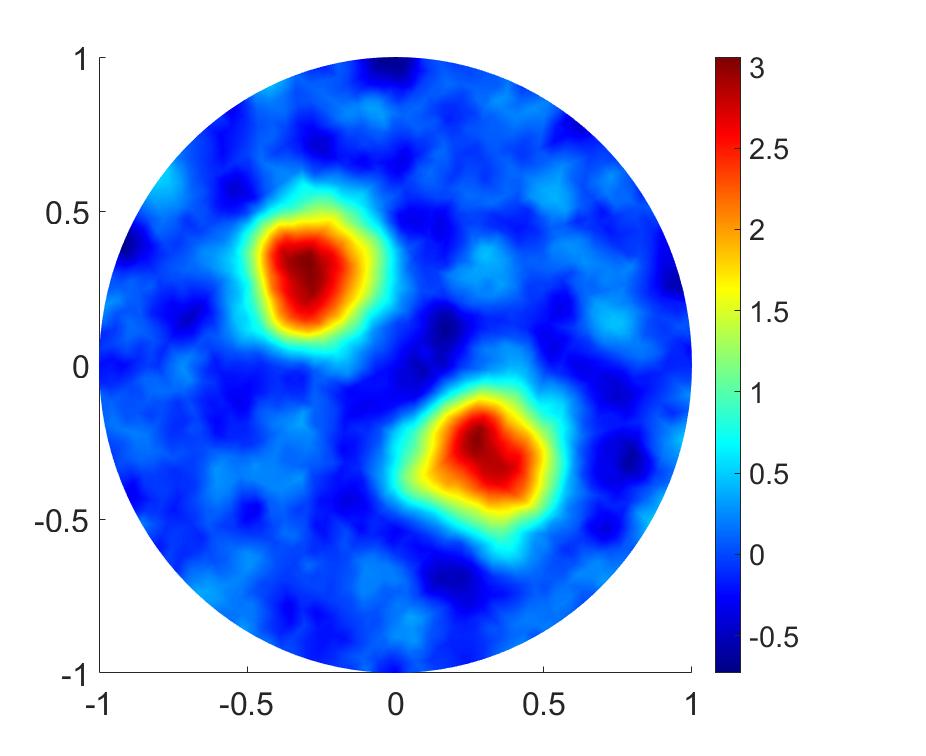}
  \end{subfigure}
  \begin{subfigure}
  \centering
  \includegraphics[width=1.6in,height=1.4in]{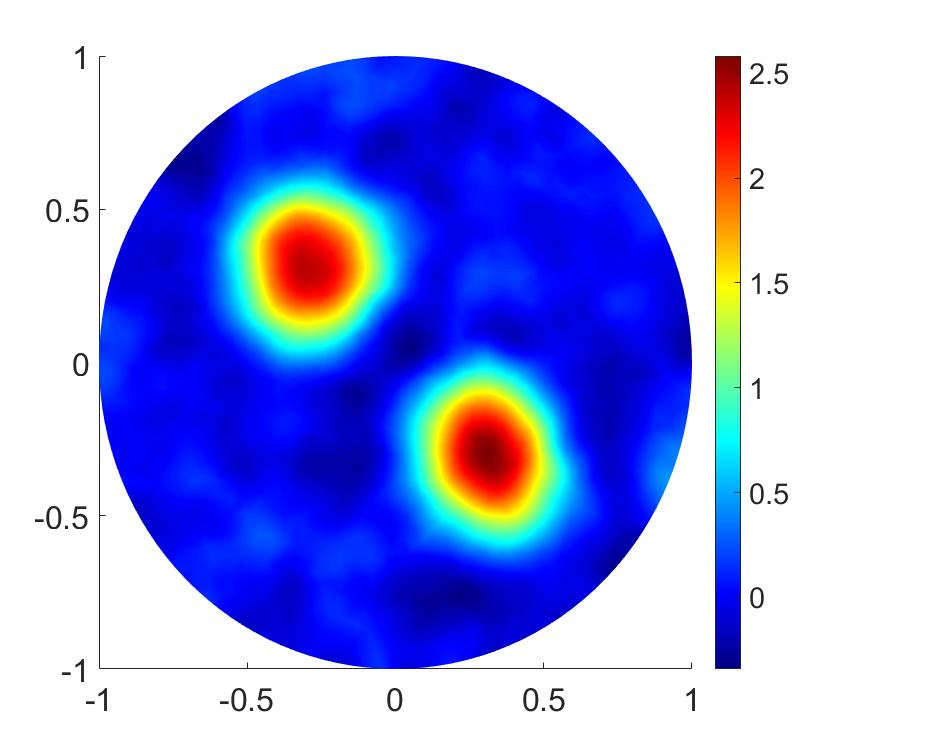}
  \end{subfigure}
  \begin{subfigure}
  \centering
  \includegraphics[width=1.6in,height=1.4in]{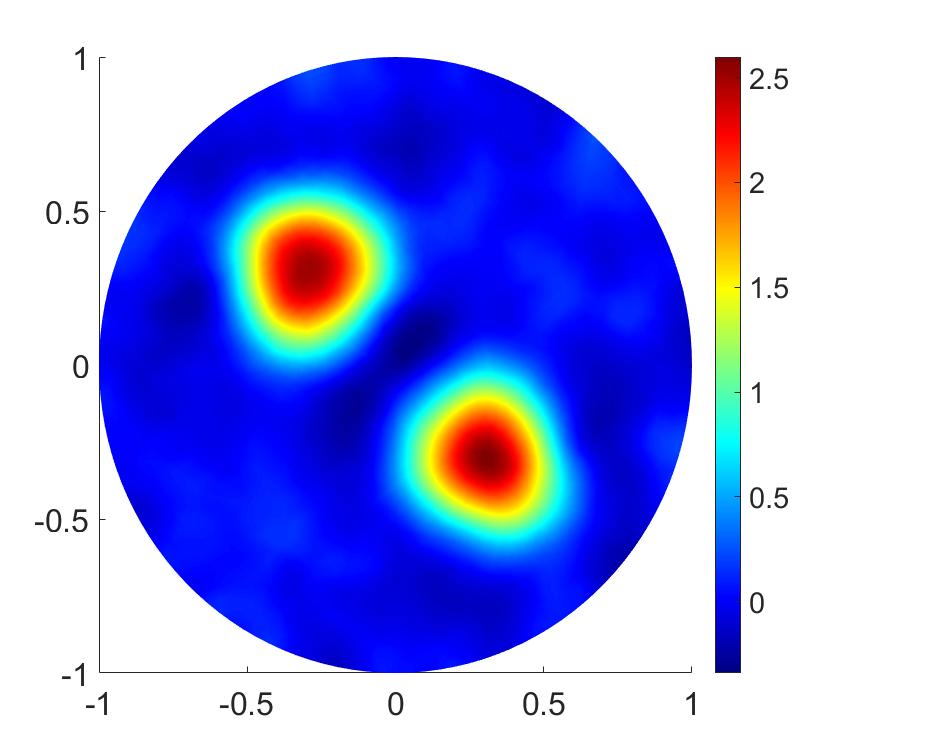}\\
  \end{subfigure}
  \begin{subfigure}
  \centering
  \includegraphics[width=1.6in,height=1.4in]{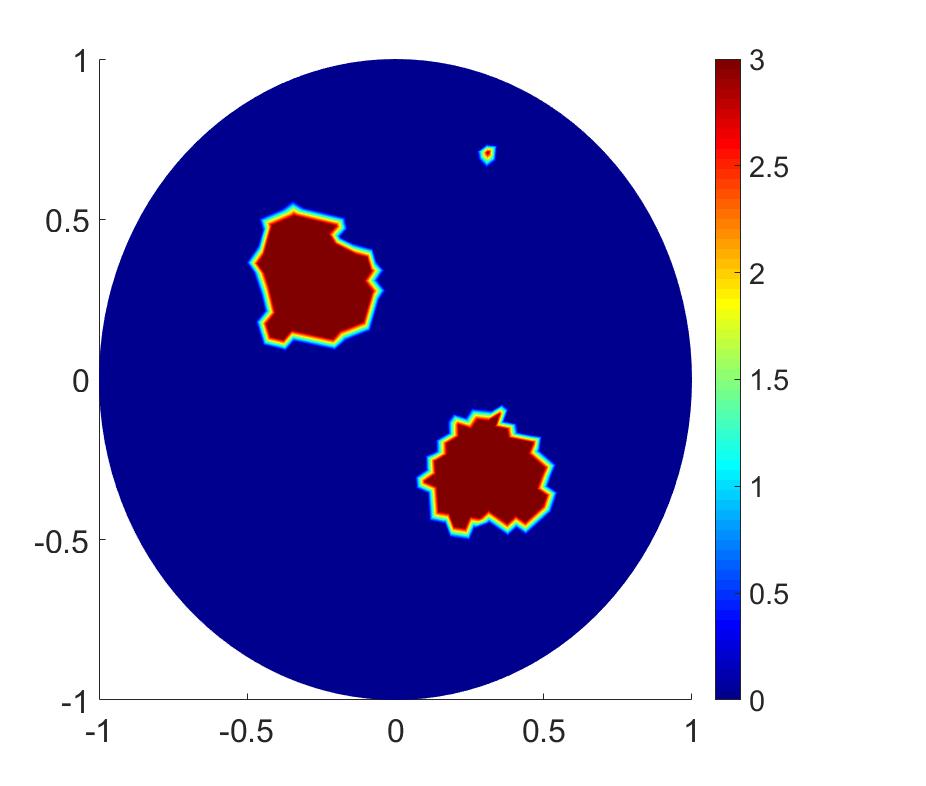}
  \end{subfigure}
  \begin{subfigure}
  \centering
  \includegraphics[width=1.6in,height=1.4in]{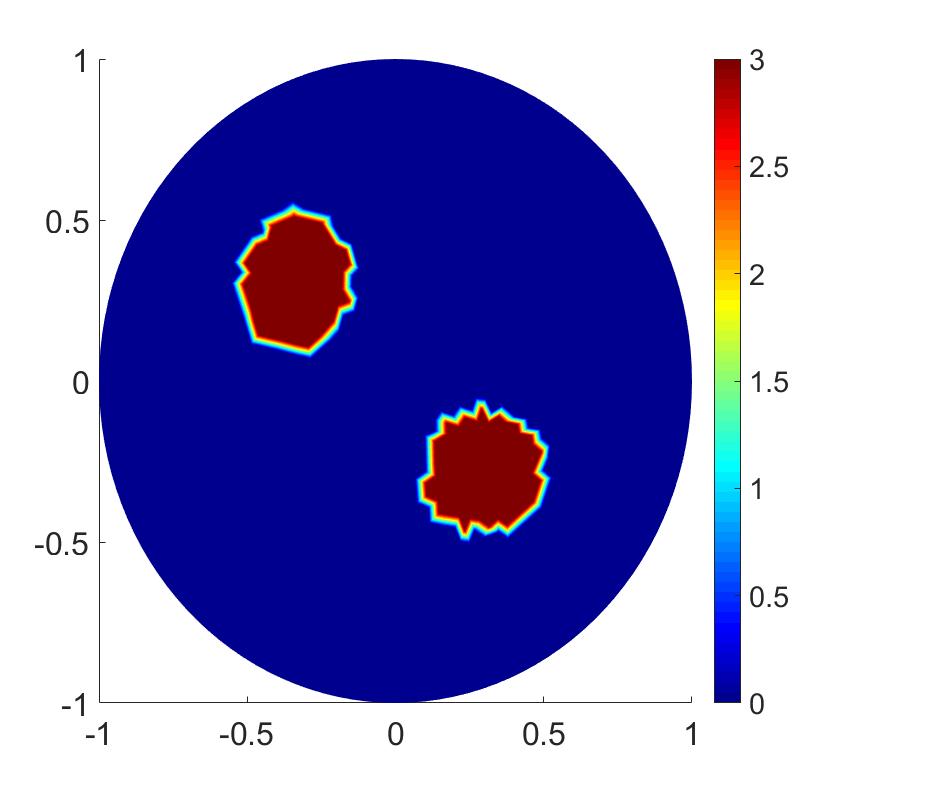}
  \end{subfigure}
  \begin{subfigure}
  \centering
  \includegraphics[width=1.6in,height=1.4in]{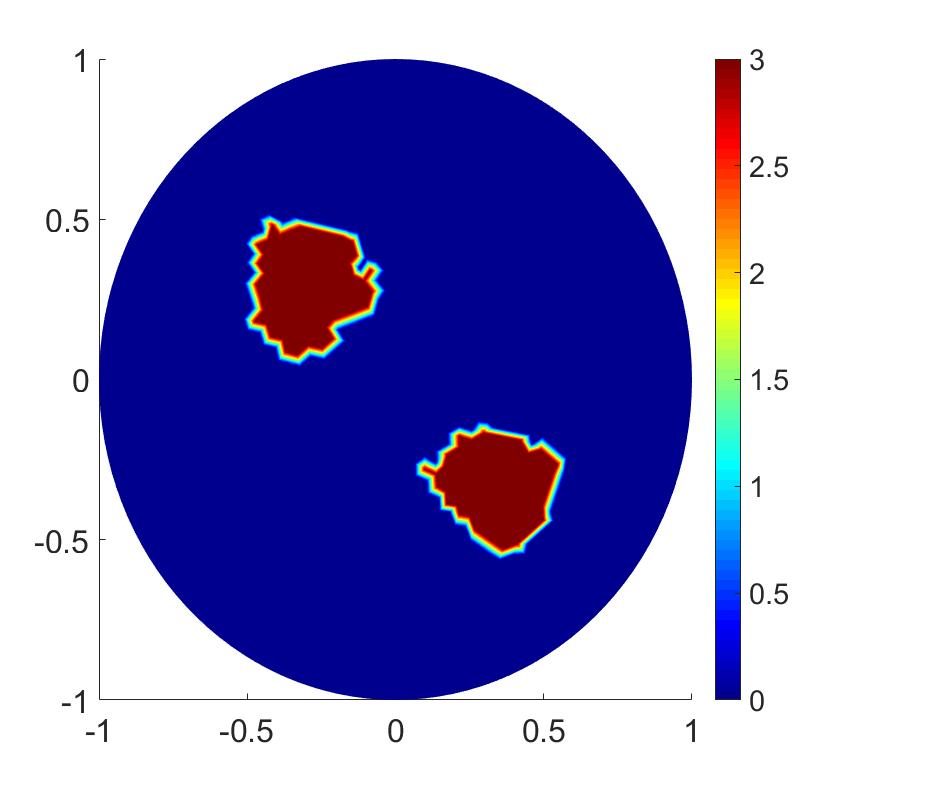}\\
  \end{subfigure}
 \caption{The reconstruction of two circle-shaped scatterers for the regular Bayesian method (top block) and the Bayesian level set method (bottom block) with $\tau=10, \frac{20}{3}, 5$, $\alpha=3$.}\label{fig10}
\end{figure}

\section*{Acknowledgments} The work of DZL is partially supported by the grants (NSFC-11601067, NSFC-117710680), and the work of XLW is partially supported by the grant (NSFC-117710680).


\bibliographystyle{AIMS}
\bibliography{Reference}

\medskip
Received xxxx 20xx; revised xxxx 20xx.
\medskip

\end{document}